\documentclass[12pt,oneside]{article}
\usepackage{amsmath,amssymb,amsfonts,amsthm}
\usepackage{color}
\usepackage[all]{xy}
\usepackage{graphicx}
\usepackage{color}
\usepackage{makeidx}
\usepackage{multirow}
\usepackage{psfrag}
\usepackage{epstopdf}
\usepackage{rotating}
\usepackage{maplestd2e}
\newcommand{\T}{{\cal T}}

\newcommand{\Real}{\mathbb R}

\newcommand{\set}[1]{\left\{#1\right\}}

\newcommand{\To}{\longrightarrow}

\newcommand {\cpp}{\mathfrak{X}(\T M)}
\newcommand {\cppp}{\mathfrak{X}(TM)}
\newcommand {\N}{\mathcal{N}}

\def\Section#1{\vspace{30truept}\addtocounter{section}{1}\setcounter{thm}{0}\setcounter{equation}{0}
{\noindent\Large\bf\arabic{section}.~~#1}\par \vspace{12pt}}

\newtheorem{thm}{Theorem}[section]
\newtheorem{cor}[thm]{Corollary}

\newtheorem{prop}[thm]{Proposition}
\theoremstyle{definition}
\newtheorem{defn}[thm]{Definition}
\newtheorem{example}[thm]{Example}
\newtheorem{rem}[thm]{Remark}

\numberwithin{equation}{section}

\newcommand\undersym[2]{\raisebox{-7pt}{\tiny$#2$}{\kern-8pt}\mbox{$#1$}}
\newcommand\undersymm[2]{\raisebox{-8pt}{\tiny$#2$}{\kern-15pt}\mbox{$#1$}}
\newcommand\overast[1]{\raisebox{10pt}{\small$\ast$}{\kern-7.5pt}\mbox{$#1$}}
\newcommand\overlind[1]{\raisebox{10pt}{\small$\overline{{\hspace{2pt}}\star}$}{\kern-7.5pt}\mbox{$#1$}}
\newcommand\overlinc[1]{\raisebox{10pt}{\tiny$\overline{{\hspace{2pt}}\circ}$}{\kern-7.5pt}\mbox{$#1$}}
\newcommand\overlina[1]{\raisebox{10pt}{\small$\overline{{\hspace{1pt}}\ast}$}{\kern-7.5pt}\mbox{$#1$}}
\newcommand\overcirc[1]{\raisebox{10pt}{\tiny{$\circ$}}{\kern-7.5pt}\mbox{$#1$}}
\newcommand\overdiamond[1]{\raisebox{10pt}{\small$\star$}{\kern-7.5pt}\mbox{$#1$}}

\begin{document}
\title{\bf{Nullity distributions associated  with Chern  connection}}
\author{\bf{ Nabil L. Youssef$^{\,1}$ and S. G.
Elgendi$^{2,3}$}}
\date{}
\maketitle                     
\vspace{-1.16cm}
\begin{center}
{$^{1}$Department of Mathematics, Faculty of Science,\\ Cairo
University, Giza, Egypt}
\end{center}

\begin{center}
{$^{2}$Department of Mathematics, Faculty of Science,\\ Benha
University, Benha,
 Egypt}
\end{center}

\begin{center}
{$^{3}$Institute  of Mathematics,  University of Debrecen,\\ Debrecen,
Hungary}
\end{center}

\begin{center}
E-mails: nlyoussef@sci.cu.edu.eg, nlyoussef2003@yahoo.fr\\
{\hspace{1.8cm}}salah.ali@fsci.bu.edu.eg, salahelgendi@yahoo.com
\end{center}

\smallskip
\vspace{1cm} \maketitle
\smallskip
{\vspace{-1.1cm}} \noindent{\bf Abstract.}
 The nullity distributions of the two curvature tensors \, $\overast{R}$ and $\overast{P}$  of the  Chern connection of a Finsler manifold  are investigated. The completeness of the nullity foliation associated with the nullity distribution $\N_{R^\ast}$  is proved. Two counterexamples are given: the first shows that  $\N_{R^\ast}$  does not coincide with the kernel distribution of \, $\overast{R}$; the second  illustrates  that   $\N_{P^\ast}$  is not  completely integrable. We give a simple  class of a non-Berwaldian Landsberg  spaces with singularities.

\vspace{7pt}


\Section{Introduction}

   Adopting the {pullback approach} to  Finsler geometry, the nullity distribution  has been investigated, for example, in  \cite{akbar.nul3., akbar.null.,ND-Zadeh}.  In 2011, Bidabad and  Refie-Rad \cite{bidabad} studied a more general distribution called $k$-nullity distribution. On the other hand, in 1982, Youssef \cite{Nabil.2,  Nabil.1} studied the nullity distributions of the curvature tensors of Barthel
 and Berwald connections, adopting the {Klein-Grifone approach }to  Finsler geometry. Moreover, Youssef  et al. \cite{ND-cartan} studied the nullity distributions associated to the  Cartan connection.

In their paper \cite{Chern}, the present  authors   investigated  the existence and uniqueness of the  Chern connection and studied the properties of its curvature tensors following the Klein-Grifone approach.
In this paper, we investigate  the nullity distributions associated with the  Chern connection. We prove the integrability and the autoparallel property of the nullity distribution      $\N_{R^\ast}$  of the Chern h-curvature \, $\overast{R}$. Moreover, we prove the completeness of the nullity foliation associated with $\N_{R^\ast}$.   We give two interesting  counterexamples.  The first shows that the nullity distribution $\N_{R^\ast}$ does not coincide with the kernel distribution of \, $\overast{R}$ ($\N_{R^\ast}$ is a proper sub-distribution of $\text{Ker}_{R^\ast}$). The second  shows that   $\N_{P^\ast}$  is not  completely integrable.  As a by-product, this allows us to give a simple  class of  non-Berwaldian Landsberg  spaces with singularities.

\Section{Notation and Preliminaries}
In this section we present a brief account of the basic concepts of Klein-Grifone's theory of Finsler manifolds. For details, we refer to references \cite{r21,r22,r27,Nabil.1}. We begin with some notational conventions.

Throughout, $M$ is a smooth manifold of finite dimension $n$. The $\Real$-algebra of smooth real-valued functions on $M$ is denoted by $C^\infty(M)$;  $\mathfrak{X}(M)$ stands for the $C^\infty(M)$-module of vector fields on $M$. The tangent bundle of  $M$ is $\pi_{M}:TM\longrightarrow M$, the subbundle of nonzero tangent vectors
 to $M$ is $\pi: \T M\longrightarrow M$. The vertical subbundle of $TTM$  is denoted by $V(TM)$. The pull-back of  $TM$ over  $\pi$ is $P:\pi^{-1}(TM)\longrightarrow \T M$. If $X\in  \mathfrak{X}(M)$, $ i_{X}$ and $\mathcal{L}_X$ denote the interior product by $X$ and the Lie derivative  with respect to $X$, respectively. The differential of $f\in C^\infty(M) $ is $df$. A vector $\ell$-form on $M$ is a skew-symmetric $C^\infty(M)$-linear map $L:(\mathfrak{X}(M))^\ell\longrightarrow \mathfrak{X}(M)$. Every vector $\ell$-form $L$ defines two graded derivations $i_L$ and $d_L$ of the Grassman algebra of $M$ such that
 $$i_Lf=0, \,\,\,\, i_Ldf=df\circ L\,\,\,\,\,\, (f\in C^\infty(M)),$$
  $$d_L:=[i_L,d]=i_L\circ d-(-1)^{\ell-1}di_L.$$

 We have the following short exact sequence of vector bundle morphisms:
\vspace{-0.1cm}
$$0\longrightarrow
 \T M \times_M TM\stackrel{\gamma}\longrightarrow T(\T M)\stackrel{\rho}\longrightarrow
\T M \times_M TM\longrightarrow 0.$$
 Here $\rho := (\pi_{\T M},\pi_\ast)$,  and $\gamma$ is defined by  $\gamma (u,v):=j_{u}(v)$, where
$j_{u}$  is the canonical  isomorphism  from $T_{\pi_{M}(v)}M$ onto $ T_{u}(T_{\pi_{M}(v)}M)$.
Then,   $J:=\gamma\circ\rho$ is a vector $1$-form  on $TM$  called the vertical endomorphism. The Liouville vector field
on $TM$ is the vector field defined by
${C}:=\gamma\circ\overline{\eta},\,\, \overline{\eta}(u)=(u,u),\, u\in TM.$

 A differential   form $\omega$ (resp. a vector form $L$) on $TM$ is semi-basic if $ i_{JX}\omega=0$ (resp. $ i_{JX}L=0$ and $ JL=0$),  for all $X\in \cppp$.  A vector $1$-form $G$ on $TM$ is called a {\textit{Grifone connection}} if it is smooth on
 $\T M$, continuous  on $TM$ and satisfies
$J G=J, \,\, G J=-J $.
The vertical and horizontal   projectors $v$\,  and
$h$ associated to $G$ are defined   by
  $$v:=\frac{1}{2}
 (I-G)\, \,\, \text{and}\,\,\, h:=\frac{1}{2} (I+G).$$

The almost complex structure determined by $G$ is the vector $1$-form $\textbf{F}$ characterized by $\textbf{F}J=h$ and $\textbf{F}h=-J$.

A Grifone connection $G$ induces the direct sum  decomposition
$$TT M=
V(TM)\oplus H(TM), \,\,\, H(TM):=\emph{{\text{Im}}} ( h).$$
The subbundle $H(TM)$ is called the $G$-horizontal subbundle of $TTM$, the module of its smooth sections will be denoted by $\mathfrak{X}^h(\T M)$.

A Grifone connection $G$ is homogeneous  if $[C,G]=0$. The torsion and the curvature of $G$ are the vector $2$-forms
$t:=\frac{1}{2} [J,G]$ and $\mathfrak{R}:=-\frac{1}{2}[h,h]$, respectively. Note that in the last three equalities the brackets mean   Fr\"{o}licher-Nijenhuis  bracket \cite{r20}.

A function  $E: TM \To \Real$ is called a {\textit{Finslerian energy}} function if it is of class $C^{1}$ on $TM$ and $C^\infty$ on $\T M$; $E(u)>0 $ if $u\in \T M$ and $E(0)=0$; $C\cdot E =2E$, i.e., $E$ is $2^+$-homogeneous; the fundamental  $2$-form     $\Omega:=dd_{J}E$   has  maximal rank.  A {\textit{Finsler manifold} } is a manifold together with a Finslerian energy. If  $(M,E)$  a  Finsler manifold, then
\begin{enumerate}
\item[\textup{(i)}] there exists a unique spray $S$ for $M$ such that $i_{S}\Omega =-dE$;
\item[\textup{(ii)}]  there exists a unique homogeneous Grifone connection on $TM$ with vanishing torsion, namely $G = [J,S]$, such that   $d_hE=0$ (\lq $G$ is conservative').
\end{enumerate}
We say that $S$ is the {\textit{canonical spray}} and $G$ is the {\textit{canonical connection}} or {\textit{Barthel connection}} of $(M,E)$.

If $(M,E)$ is  a  Finsler manifold, then the map $\overline{g}$ given by
$$\overline{g}(J X,J Y):=\Omega(JX,Y); \,\,  \,\,  X, Y \in  \cpp$$
is  a  metric tensor  on $V(TM)$. It   can be extended to a metric tensor $g$ on $T(TM)$ by

 \begin{equation}\label{metricg}
 g(X,Y):=\overline{g}(JX,JY)+\overline{g}(vX,vY)=\Omega(X,\textbf{F}Y).
\end{equation}

Now we recall three famous covariant derivative  operators on a Finsler manifold, called also  \lq connections'. They are the \textit{Berwald connection }\, $\overcirc{D}$, the \textit{Cartan  connection} ${D}$ and the \textit{Chern  connection }\,  $\overast{D}$, given by

\begin{equation}\label{berwaldconn.}
  \overcirc{D}_{JX}JY=J[JX,Y],\quad
\overcirc{D}_{hX}JY=v[hX,JY],\quad
  \overcirc{D}\textbf{F}=0;
\end{equation}
\begin{equation}\label{cartanconn.}
  D_{JX}JY=\overcirc{D}_{JX}JY+\mathcal{C}(X,Y),\quad
  D_{hX}JY=\overcirc{D}_{hX}JY+\mathcal{C}'(X,Y),\quad
  {D}\textbf{F}=0;
\end{equation}
\begin{equation}\label{chernconn.}
  \overast{D}_{JX}JY=J[JX,Y],\quad
 \overast{D}_{hX}JY=v[hX,JY]+\mathcal{C}'(X,Y),\quad
  \overast{D}\textbf{F}=0,
\end{equation}
($X,Y\in \mathfrak{X}(\T M)$). In the formulas (\ref{cartanconn.}) and (\ref{chernconn.})
 $\mathcal{C}$ is the \textit{Cartan tensor},  $\mathcal{C}'$ is the \textit{Landsberg tensor} of $(M,E)$. For their definition, see \cite{r22}, p. 329.
 The tensors  ${\mathcal{C}}$ and $\mathcal{C}'$   are symmetric,  semi-basic and for arbitrary  semispray $S$ on $TM$, we have
  \begin{equation}\label{c(s)}
  {\mathcal{C}}(X,S)=\mathcal{C}'(X,S)=0.
  \end{equation}

  Let \, $\overast{R}$ and \,  $\overast{P}$ be the h-curvature and  the hv-curvature  of \, $\overast{D}$, respectively. We list some important identities  from \cite{Chern}, which will be needed in the sequel. Below $X$, $Y$, $Z$, $W$ are vector fields, $S$ is a semispray on $\T M$.
  \begin{equation}\label{chern.[]}
  [hX,hY]=h(\,\overast{D}_{hX}Y-\,\overast{D}_{hY}X)-\mathfrak{R}(X,Y);
\end{equation}
  \begin{equation}\label{eq.1}
  \overast{R}(X,Y)Z=R(X,Y)Z-\mathcal{C}(\textbf{F}\mathfrak{R}(X,Y),Z),
  \end{equation}
  where $R$ is the h-curvature of $D$;
  \begin{equation}\label{eq.2}
  \overast{P}(X,Y)Z=\overcirc{P}(X,Y)Z-(\,\,\overast{D}_{JY}\mathcal{C}')(X,Z),
  \end{equation}
  where \,$\overcirc{P}$ is the hv-curvature of \,$\overcirc{D}$;
   \begin{equation}\label{eq.3}
  \overast{R}(X,Y)S=\mathfrak{R}(X,Y);
  \end{equation}
   \begin{equation}\label{eq.4}
  \overast{P}(X,Y)S=\overast{P}(S,Y)X=\mathcal{C}'(X,Y),\,\,\,\overast{P}(X,S)Z=0;
  \end{equation}
  \begin{equation}\label{eq.5}
  \mathfrak{S}_{X,Y,Z}\, \{\,\overast{R}(X,Y)Z\}=0;
  \end{equation}
   \begin{equation}\label{eq.6}
\mathfrak{S}_{X,Y,Z}\,\{(\,\overast{D}_{hX}\,\mathfrak{R})(Y,Z)\}=\mathfrak{S}_{X,Y,Z}\{\,
  \mathcal{C}'(\textbf{F}\mathfrak{R}(X,Y),Z)\};
  \end{equation}
   \begin{equation}\label{eq.7}
 \mathfrak{S}_{X,Y,Z}\{\,(\,\overast{D}_{hX}\,\overast{R})(Y,Z)\}=\mathfrak{S}_{X,Y,Z}\{\, \overast{P}(X,\textbf{F}\mathfrak{R}(Y,Z))\};
  \end{equation}
   \begin{equation}\label{eq.8}
 (\,\overast{D}_{hX}\,\overast{P})(Y,Z)-(\,\overast{D}_{hY}\,\overast{P})(X,Z)+(\,\overast{D}_{JZ}\,\overast{R})(X,Y)
  =\overast{P}(X,\textbf{F}\mathcal{C}'(Y,Z))
  -\overast{P}(Y,\textbf{F}\mathcal{C}'(X,Z));
  \end{equation}
  If $\mathfrak{R}=0$, then
  \begin{equation}\label{eq.9}
 \overast{R}(X,Y,Z,W)=\,\overast{R}(Z,W,X,Y),
  \end{equation}
  where \,$ \overast{R}(X,Y,Z,W):=g(\,\overast{R}(X,Y)Z,JW)$.


\Section{Nullity distribution of the Chern h-curvature}
In this section, we investigate the nullity distribution of the Chern connection.  It should be noted that   the nullity distributions of  the Barthel, Berwald and Cartan  connections  have already been studied in \cite{Nabil.2,Nabil.1,ND-cartan}, respectively. First, we  study the nullity
 distribution of the h-curvature tensor.

\begin{defn}\label{nul.chern} Let \, $\overast{R}$ be the h-curvature tensor of the Chern connection.
The \textnormal{nullity space} of \, $\overast{R}$ at a point $z\in TM$ is the subspace of $H_z(TM)$ defined by
$${\N}_{R^\ast}(z):=\{v\in H_z(TM)| \,\,  \overast{R}_z(v,w)=0, \, \,\text{for all}\,\, w\in H_z(TM)\}.$$
The dimension of ${\N}_{R^\ast}(z)$, denoted by  ${\mu}_{{R^\ast}}(z)$, is the  \textnormal{nullity index} of \, $\overast{R}$ at $z$.
If the  nullity index ${\mu}_{R^\ast}$ is constant,
then the map ${\mathcal{N}}_{R^\ast}:z\mapsto {\N}_{R^\ast}(z) $ defines a distribution $\N_{R^\ast}$ of
rank ${\mu}_{R^\ast}$, called the \textnormal{nullity distribution} of\, $\overast{R}$.
Any  smooth section in the nullity distribution $\N_{R^\ast}$ is called  \textnormal{a nullity vector field}.
We denote by $\Gamma({\N}_{R^\ast})$ the $C^\infty(TM)$-module of the nullity vector fields.
We shall assume that  $\mu_{R^\ast}\neq 0$ \,and \,$\mu_{R^\ast}\neq n$.
 \end{defn}

Let  $\N_{R^\ast}(x):=\pi_\ast(\N_{R^\ast}(z))$ if $\pi(z)=x$. Then   $\N_{R^\ast}(x)$ isomorphic to $\N_{R^\ast}(z)$ via the isomorphism  $\pi_\ast\!\upharpoonright_{H_z(TM)}$.

 \begin{defn}\label{ker.chern}
The kernel of \, $\overast{R}$ at the point $z\in TM$  is defined by
$$\emph{\text{Ker}}_{R^\ast}(z):=\{u\in H_z(TM)| \, \, {\overast{R}_z}(v,w) u=0, \, \text{for all}\, \, v,w\in H_z(TM)\}. $$
\end{defn}
We have  $\emph{\text{Ker}}_{R^\ast}(x)=\pi_\ast(\emph{\text{Ker}}_{R^\ast}(z))$; $x=\pi(z)$.
\begin{prop}\label{chern.nul} The nullity distribution $\N_{R^\ast}$ has the following properties:
~\par
\begin{enumerate}
   \item[\textup{(1)}] $\N_{R^\ast}\neq \phi$ and $\emph{\text{Ker}}_{R^\ast}\neq \phi$.

  \item[\textup{(2)}] $\N_{R^\ast}\subseteq \mathcal{N}_\mathfrak{R}$, where $\mathcal{N}_\mathfrak{R}$
  is the nullity distribution of the curvature $\mathfrak{R}$ of the Barthel connection.

  \item[\textup{(3)}]  $ \N_{R^\ast}\subseteq \emph{\text{Ker}}_{R^\ast}$.

  \item[\textup{(4)}] If the canonical spray $S$ belongs to $\Gamma(\N_{R^\ast})$, then $\mathfrak{R}=0$.

  \item[\textup{(5)}]  If $X\in \Gamma(\N_{R^\ast})$, then $[C,X]\in \Gamma(\N_{R^\ast})$
and,  consequently,
$[C,X]\in \Gamma(\N_{\mathfrak{R}})$.
\end{enumerate}

\end{prop}
\begin{proof}
 (2) Let $X$ be  a  nullity  vector field. Using  (\ref{eq.3}),  we have 
\begin{eqnarray*}
               X\in \Gamma(\N_{R^\ast}) &\Longrightarrow& \overast{R}(X,Y)Z=0\quad \text{for all}\, \, Y,Z \in \cppp \\
                &\Longrightarrow& \overast{R}(X,Y)S=0\quad \,\text{for all}\, \, Y \in \cppp\\
                &\Longrightarrow& \mathfrak{R}(X,Y)=0\quad\,\,\,\, \text{for all}\, \, Y \in \cppp\\
                 &\Longrightarrow&  X\in \Gamma(\N_{\mathfrak{R}}).
             \end{eqnarray*}
\noindent (3) Let $Z\in \Gamma(\mathcal{N}_{R^\ast} )$,
then,  by  (\ref{eq.5}), we have
 $\mathfrak{S}_{X,Y,Z}\{\,\overast{R}(X,Y)Z\}=0.$
Since \, $\overast{R}(Y,Z)X=\overast{R}(Z,X)Y=0$,
 then the result follows.\\

\noindent (4) This is an immediate consequence of  (\ref{eq.3}).\\

\noindent (5)  Let  $X\in \Gamma(\N_{R^\ast})$. Since  $\,\overast{D}_C\,\overast{R}=0$ \cite{Chern} , we get
 $(\,\overast{D}_C\,\overast{R})(X,Y)=0,$
which leads to
$\overast{R}(\,\overast{D}_CX,Y)=0.$
Using   (\ref{chernconn.}), we have $\overast{R}([C,X],Y)=0$. By the homogeneity of $h$,  $[C,h]=0$, from which  $[C,hX]=h[C,X]$. That is,
 $[C,hX]$ is horizontal.
 Hence, $[C,X]\in \Gamma(\N_{R^\ast})$. Consequently, by (2),
$[C,X]\in \Gamma(\mathcal{N}_{\mathfrak{R}})$.
\end{proof}

It is important to note that the reverse inclusion in the property  (3) of Proposition \ref{chern.nul} is not true;  that is, ${\text{Ker}}_{R^\ast}\not\subset \mathcal{N}_{R^\ast}$. This is shown by the next example  in which the calculations are performed by using \cite{CFG}.

\bigskip

\begin{example}
 Let $M=\{(x^1,x^2,x^3,x^4)\in \mathbb{R}^4| x^2>0\}$ and \\
 $U=\{(x^1,...,x^4;y^1,...,y^4)\in \mathbb{R}^4 \times \mathbb{R}^4: \,y^1\neq 0, y^2\neq 0\}\subset TM$. Define   $F$  on $U$ by
 $$F(x,y) := \, ({{{{( x^2)}}^{2}{{(y^1)}}^{4}+{{( y^2)}}^{4}+{{( y^3)}}^{4}+{{( y^4)}}^{4}}})^{1/4}.$$

According to \cite{appl.}, the  nullity distribution of the Cartan h-curvature $R$ of $(M,F)$ is
\begin{equation}\label{null.1}
\N_R=\{sh_3+th_4\in \mathfrak{X}^h(\T M)| \,s,t\in \mathbb{R}\}
\end{equation}
 and the kernel distribution $\ker_R$ of $R$ is
 \begin{equation}\label{ker1}
\begin{split}
 \ker_R=\left\{s\left(\frac{y^1}{y^2}h_1+h_2+\frac{x^2{(y^1)}^4+{(y^2)}^4+2{(y^3)}^4+2{(y^4)}^4}{y^2{(y^4)}^3}h_4\right)\right. &\\ \left.+t\left(h_3-\frac{{(y^3)}^3}{{(y^4)}^3}h_4\right)\in \mathfrak{X}^h(\T M)|s,t\in \mathbb{R}\right\},
 \end{split}
 \end{equation}
where $h_i:=\frac{\partial}{\partial x^i}-N^m_i\frac{\partial}{\partial y^m}$ form a basis of $\mathfrak{X}^h(\T M)$.

\bigskip

Now, by the NF-package \cite{CFG}, we can perform the following calculations.

\smallskip
\begin{maplegroup}
\begin{Maple Normal}{
\textbf{Chern h-curvature\, \overast{R}:}}\end{Maple Normal}

\end{maplegroup}
\begin{maplegroup}
\begin{mapleinput}
\mapleinline{active}{1d}{show(Rchern[i, -h, -j, -k]); }{\[\]}
\end{mapleinput}
\mapleresult
\begin{maplelatex}
\mapleinline{inert}{2d}{}{}{\[ {\it Rchern}^{{\it x1} }_{{\it x1 x1 x2} }=\frac{1}{18}\frac{4{\it y2}^{4}+{\it x2}^{2}{\it y1}^{4}}{{\it x2}^{2}{\it y1} {\it y2}^{3}} \hspace{2.3cm} {\it Rchern}^{{\it x1} }_{{\it x2 x1 x2} }=-\frac{1}{9}\frac{4{\it y2}^{4}+{\it x2}^{2}{\it y1}^{4}}{{\it x2}^{2}{\it y2}^{4}}\]}
\end{maplelatex}
\mapleresult
\begin{maplelatex}
\mapleinline{inert}{2d}{}{}{\[ {\it Rchern}^{{\it x2} }_{{\it x1 x1 x2} }=\frac{1}{9}\frac{4{\it y1}{2}{\it y2}^{4}+{\it x2}^{2}{\it y1}^{6}}{{\it y2}^{6}}
\hspace{2cm} {\it Rchern}^{{\it x2} }_{{\it x2 x1 x2} }=-\frac{1}{18}\frac{4{\it y1}^{3}{\it y2}^{4}+{\it x2}^{2}{\it y1}^{7}}{{\it y2}^{7}}\]}
\end{maplelatex}
\end{maplegroup}
\vspace{7pt}
\begin{maplegroup}
\begin{Maple Normal}{
\textbf{Nullity distribution of\, $\overast{R}$:}}\end{Maple Normal}

\end{maplegroup}
\begin{maplegroup}
\begin{mapleinput}
\mapleinline{active}{1d}{definetensor(RchernW[h, -i, -k] = Rchern[h, -i, -j, -k]*W[j]); }{\[\]}
\end{mapleinput}
\end{maplegroup}
\begin{maplegroup}
\begin{mapleinput}
\mapleinline{active}{1d}{show(RchernW[h, -i, -k]); }{\[\]}
\end{mapleinput}
\mapleresult

\begin{maplelatex}
\mapleinline{inert}{2d}{}{}{\[ {\it RchernW}^{{\it x1} }_{{\it x2 x2} }=-\frac{1}{9}\frac{\left(4{\it y2}^{4}+{\it x2}^{2}{\it y1}^{4}\right)W^{{\it x1} }}{{\it x2}^{2}{\it y2}^{4}}\hspace{1cm} {\it RchernW}^{{\it x2} }_{{\it x1 x1} }=-\frac{1}{9}\frac{\left(4{\it y1}^{2}{\it y2}^{4}+{\it x2}^{2}{\it y1}^{6}\right)W^{{\it x2} }}{{\it y2}^{6}}\]}
\end{maplelatex}
\mapleresult
\end{maplegroup}

\bigskip

Putting ${\it RchernW}^{\it x1 }_{\it x2 x2 }=0$ and  ${\it RchernW}^{\it x2 }_{\it x1 x1 }=0$, then we have a system of algebraic equations. The NF-package yields the following solution: $W^1=W^2=0,   W^3=s, W^4=t$, where $s,t\in \mathbb{R}$. Then,  the  nullity distribution is
\begin{equation}\label{nullch.1}
\N_{R^\ast}=\{sh_3+th_4\in \mathfrak{X}^h(\T M)|\,s,t\in \mathbb{R}\}.
\end{equation}

\begin{Maple Normal}{
\begin{Maple Normal}{
\textbf{Kernel distribution of\, $\overast{R}$:}}\end{Maple Normal}

}\end{Maple Normal}

\begin{maplegroup}
\begin{mapleinput}
\mapleinline{active}{1d}{definetensor(RchernZ[h, -j, -k] = Rchern[h, -i, -j, -k]*Z[i]); }{\[\]}
\end{mapleinput}
\end{maplegroup}
\begin{maplegroup}
\begin{mapleinput}
\mapleinline{active}{1d}{show(RchernZ[h, -j, -k]); }{\[\]}
\end{mapleinput}
\mapleresult
\begin{maplelatex}
\mapleinline{inert}{2d}{}{\[ {\it RchernZ}^{{\it x1} }_{{\it x1 x2} }=\frac{1}{18}\frac{\left(4{\it y2}^{4}+{\it x2}^{2}{\it y1}^{4}\right)Z^{{\it x1} }}{{\it x2}^{2}{\it y1} {\it y2}^{3}}-\frac{1}{9}\frac{\left(4{\it y2}^{4}+{\it x2}^{2}{\it y1}^{4}\right)Z^{{\it x2} }}{{\it x2}^{2}{\it y2}^{4}}\]}
\end{maplelatex}
\mapleresult
\end{maplegroup}

\bigskip

Putting ${\it RchernZ}^{\it x1 }_{\it x1 x2 }=0$, we get $Z^1=\frac{2y1}{y2}r,   Z^2=r, Z^3=s, Z^4=t;\,r,s,t\in \mathbb{R}$. Then,  the kernel distribution  $ \text{Ker}_{R^\ast}$ is
\begin{equation}\label{kerch.1}
\text{Ker}_{R^\ast}=\set{r\left(\frac{2y1}{y2}h_1+h_2\right)+sh_3+th_4\in \mathfrak{X}^h(\T M)|\,r,s,t\in \mathbb{R}}.
\end{equation}
Equations (\ref{nullch.1}) and (\ref{kerch.1}) show that $\text{Ker}_{R^\ast}$ can not be a sub-distribution of $\N_{R^\ast}$.
\end{example}
\begin{thm}\label{coincide}
The  nullity distribution $\N_{R^\ast}$  of the  Chern h-curvature and the  nullity distribution $\N_R$ of the Cartan h-curvature
  coincide.
\end{thm}

\begin{proof} Let $X\in\Gamma(\N_{R^{\ast}})$. Then, by   (\ref{eq.1}) and  Proposition \ref{chern.nul} (2),
$X\in\Gamma(\N_R)$. Hence $\N_{R^{\ast}}$ is a subset of $ \N_R$. Conversely, let $X\in\Gamma(\N_R)$. Then, by (\ref{eq.1}) and by $\N_R\subset\N_\mathfrak{R}$ \cite{ND-cartan}, we get $X\in\Gamma(\N_{R^{\ast}})$, whence, $\N_R\subset \N_{R^\ast}$.
\end{proof}

\begin{rem}
 The above example shows that $\N_{R^{\ast}}\subset\text{Ker}_{R^\ast}$ and the reverse inclusion is false by (\ref{nullch.1}), (\ref{kerch.1}).
 It also shows that although $\N_{R^{\ast}}= \N_R$ (see (\ref{null.1}) and (\ref{nullch.1})),  $\text{Ker}_{R^\ast}\neq\text{Ker}_{R}$  by (\ref{ker1}), (\ref{kerch.1}).
 In view of the above theorem, the reverse inclusion in { (2)} of Proposition \ref{chern.nul} is not true either: $\mathcal{N}_\mathfrak{R}\not\subset \mathcal{N}_R=\mathcal{N}_{R^\ast}$ \cite{appl.}.
 \end{rem}
\begin{defn}
The conullity space of the h-curvature tensor at $z$,
denoted by ${\N_{R^\ast}\!}^\perp(z)$, is the orthogonal complement of $\N_{R^\ast}$ in $H_z(TM)$,  where the orthogonality  is taken with respect to the metric $g$ defined by (\ref{metricg}).
\end{defn}

\begin{prop}
For each point $z\in TM$, either $\mu_{R^\ast}(z)=n$ or  $\mu_{R^\ast}(z)\leq n-2$. Consequently, $\dim{\emph{\text{Ker}}_{R^\ast}}>n-2$.
\end{prop}
\begin{proof}
If $\mu_{R^\ast}(z)\neq n$, then there is a non-zero horizontal vector $v\notin \N_{R^\ast}(z) $. It follows that there is a vector $w\in H_z(TM)$
 such that \, $\overast{R}_z(w,v)\neq 0$ and so \,  $\overast{R}_z(v,w)\neq 0$. Then $v,w \notin \N_{R^\ast}(z)$ and hence $v,w \in {\N_{R^\ast}\!}^\perp(z)$. By the antisymmetry of \, $\overast{R}$,  the vectors $v$ and $w$ are  independent.  Thus, $\dim {\N_{R^\ast}\!}^\perp(z)\geq 2$. Consequently, $\mu_{R^\ast}(z)\leq n-2$.
\end{proof}

\begin{prop}
If \,$\mathfrak{R}=0$, then  \em{Im }\!\!$(\,\overast{R}) = (J  \N_{R^\ast} )^\perp$.
Consequently, \text{rank }\!$(\,\overast{R})=n-{\mu}_{R^\ast}$.
\end{prop}

\begin{proof}
For all $X\in \Gamma({\mathcal{N}}_{R^\ast})$ and $Y,Z,W\in \mathfrak{X}^h(\T M)$, we have
\begin{eqnarray*}
  g(\,\overast{R}(Y,Z)W,JX) &=& \overast{R}(Y,Z,W,X)\\
                          &=& \overast{R}(W,X,Y,Z) \,\,\,( \text{by (\ref{eq.9}))}\\
                          &=& -\overast{R}(X,W,Y,Z)\\
                          &=& -g(\,\overast{R}(X,W)Y,JZ)\\
                          &=& 0 \,\,\,( \text{since } X \text{ is a nullity vector field}),
\end{eqnarray*}
as wanted.
\end{proof}

As a direct consequence  of  Theorem \ref{coincide} and the fact that $\N_R$ is completely integrable \cite{ND-cartan}, we have the
following  result.

\begin{cor}\label{c.i.chern}  Let $\mu_{R^\ast}$ be constant on an open subset $U$ of $TM$.
The nullity distribution $z\mapsto \mathcal{N}_{R^\ast}(z)$ is completely integrable on $U$.
\end{cor}

According to the Frobenius theorem,  there exists a foliation of $M$ by $\mu_{R^\ast}$-dimensional maximal connected submanifolds  as  leaves, such that the nullity space at a point $x\in M$ is the tangent space to the leaf at $x$. We call the foliation induced by the nullity distribution $\N_{R^\ast}$ the nullity foliation and denote it again by $\N_{R^\ast}$. So, by  Corollary \ref{c.i.chern}, we have the following result.

\begin{thm}\label{chern.autoparallel} The leaves of the nullity foliations $\mathcal{N}_{R^\ast}$ and  $\mathcal{N}_\mathfrak{R}$
are auto-parallel submanifolds with respect to the Chern connection.
\end{thm}
\begin{proof}
The fact that $\N_{R^\ast}$ is auto-parallel  with
 respect to Chern connection can be proved in a similar manner as the analogous result in \cite{ND-cartan}.

On the other hand, the integrability of the nullity distribution $\N_\mathfrak{R}$ of the curvature of Barthel  connection
 has been proved in  \cite{Nabil.2}. We show that if $X,Y\in \Gamma(\N_\mathfrak{R})$, then $\,\overast{D}_XY\in  \Gamma(\N_\mathfrak{R})$.
  By  (\ref{eq.6}), we have
$$\mathfrak{S}_{X,Y,Z}\{(\,\overast{D}_X\mathfrak{R})(Y,Z)\}=\mathfrak{S}_{X,Y,Z}\{\,
  \mathcal{C}'(Z,\textbf{F}\mathfrak{R}(X,Y))\}.$$
Since $X,Y\in \Gamma(\N_\mathfrak{R})$,
$\mathfrak{S}_{X,Y,Z}\{(\,\overast{D}_X\mathfrak{R})(Y,Z)\}=0.$
Consequently,  $\mathfrak{R}(\,\overast{D}_XY,Z)=0 $ for every vector field   $Z\in \cppp$ and $\,\overast{D}_XY\in  \Gamma(\N_\mathfrak{R})$.
\end{proof}

Due to the torsion-freeness  of the Levi-Civita connection in Riemannian geometry,  the two  concepts \lq autoparallel submanifold\rq \, and  \lq totally geodesic
submanifold\rq \, coincide  \cite{kobayashi}. This is not true in Finsler geometry.  However,  every auto-parallel submanifold   is totally geodesic \cite{E.cartan}. So,
 we have:
\begin{cor}\label{leaves}
The leaves of the nullity foliations $\N_\mathfrak{R}$ and $\N_{R^{\ast}}$ are totally geodesic submanifolds with respect to the Chern connection.
\end{cor}
\begin{thm}\label{coincidenulker}
If $\mathfrak{R}=0$, then the two distributions   $\N_{R^\ast}$ and $\text{Ker}_{R^\ast}$ coincide.
\end{thm}
\begin{proof} By Proposition \ref{chern.nul} (3), we always have $\N_{R^\ast}\subset \text{Ker}_{R^\ast}$. Let $X\in  \Gamma(\text{Ker}_{R^\ast})$ and let $Y,Z,W$ be vector fields on $T M$, then by  (\ref{eq.9}), we have
\begin{eqnarray*}
  \overast{R}(Y,Z)X=0  &\Longrightarrow& g(\overast{R}(Y,Z)X,JW)=0  \\
   &\Longrightarrow&   \overast{R}(Y,Z,X,W)=0\\
    &\Longrightarrow&   \overast{R}(X,W,Y,Z)=0\\
    &\Longrightarrow&   g(\,\overast{R}(X,W)Y,JZ)=0\\
    &\Longrightarrow&   \overast{R}(X,W)Y=0\\
    &\Longrightarrow&   X\in \Gamma(\N_{R^\ast}),
\end{eqnarray*}
thus $\text{Ker}_{R^\ast}\subset \N_{R^\ast}$.
\end{proof}
\begin{thm}
Let $(M, E)$ be a complete Finsler manifold and $U$ the open subset
of M on which $\mu_{R^\ast}$ takes its minimum. If\, $\mathfrak{R}$ vanishes, then every integral manifold of the nullity
foliation $\mathcal{N}_{R^\ast}$  in $U$ is  complete.
\end{thm}
\begin{proof}The proof is inspired by \cite{akbar.null.}, taking into account the fact that the two spaces $\N_{R^\ast}(z)$ and $\N_{R^\ast}(x)$, $x=\pi(z)$,  are isomorphic.
Let $N$ be an integral manifold of the nullity foliation $\mathcal{N}_{R^\ast}$ in $U$. To prove that $N$ is complete, it suffices to show that every geodesic $\gamma : [0, c) \rightarrow N $ on $N$ can be extended to a geodesic $ \widetilde{\gamma}: [0,\infty)\rightarrow N $ on $N$. Suppose that such a geodesic extension $ \widetilde{\gamma}$ does not exist. As $N$ is totally geodesic, by Corollary \ref{leaves}, $\gamma$ is a geodesic on M and thus has  a geodesic extension $ \widetilde{\gamma}~:~ [0,\infty)\rightarrow~M $ such that $\gamma=\widetilde{\gamma}\cap N$.  It follows that $p:=\widetilde{\gamma}(c)\notin U$. Let $p_0:=\gamma(0)=\widetilde{\gamma}(0)$ and set $r_0:=\mu_{R^\ast}(p_0)$, the dimension of the nullity space $\mathcal{N}_{R^\ast}(p_0)$. Since  $\mu_{R^\ast}$  is positive and minimal  on $U$, then  $\mu_{R^\ast}(p)>r_0>0$. Now, consider a basis $B=\{e_1,...,  e_{r_0},e_{r_0+1},...,e_n\}$ for $T_{p_0}M$ such that $\{e_1, ..., e_{r_0}\}$ is a basis for $\N_{R^\ast}(p_0)$ and $e_1$ is tangent to $\gamma$ at $p_0=\gamma(0)$. Using the system of differential
equations $$\frac{\overast{D} F_i}{dt}=0,\quad F_i(0)=e_i, \quad i=1,2,...,n,$$
the basis $B$ can be translated into a parallel frame $(F_1, ..., F_{r_0},F_{r_0+1},...,F_n)$ along $\widetilde{\gamma}$. Then $(F_1, ..., F_{r_0})$ is a basis for the nullity space at every point $\widetilde{\gamma}(t)$ in $U\cap V$ for some neighborhood $V$ of $\widetilde{\gamma}(t)$ on $M$. Since $\mu_{R^\ast}(p)>r_0$, there is a vector field $F_a$ along $\widetilde{\gamma}$, for a fixed integer $a$ in the range $r_0+1,...,n$, such that for every $t\in[0,c)$, we have
\begin{equation}\label{contradiction-1}
F_a(\gamma(t))\notin\mathcal{N}_{R^\ast}(\gamma(t)), \,\,\, \,\,  F_a(p)\in \mathcal{N}_{R^\ast}(p).
\end{equation}

Now, let $\widehat{\widetilde{\gamma}}$ be the natural lift of $\widetilde{\gamma}$ to $\T M$ and $\{\widehat{F}_1, ..., \widehat{F}_{r_0},\widehat{F}_{r_0+1},...,\widehat{F}_n\}$ the basis of $H_{\widehat{\widetilde{\gamma}}(t)}TM$ such that $\pi_\ast(\widehat{F}_i)=F_i$. Let $\phi^h_{ijk}$ be the functions defined by
\begin{equation}\label{complete}
\overast{R}(\widehat{F}_i,\widehat{F}_j)\widehat{F}_k=\phi^h_{ijk}\,
\frac{\partial}{\partial y^h}.
\end{equation}
 By  (\ref{eq.7}), taking into account  that $\mathfrak{R}=0$, we have
$$(\,\overast{D}_{hX}\,\overast{R})(Y,Z)+(\,\overast{D}_{hY}\,\overast{R})(Z,X)+(\,\overast{D}_{hZ}\,\overast{R})(X,Y)=0.$$
Plugging $\widehat{F}_1$, $\widehat{F}_i$ and  $\widehat{F}_j$  instead of $X$, $Y$ and $Z$, where $i,j=r_0+1,...,n$,  we get
$$(\,\overast{D}_{\widehat{F}_1}\,\overast{R})(\widehat{F}_i,\widehat{F}_j)+(\,\overast{D}_{\widehat{F}_i}\,\overast{R})(\widehat{F}_j,\widehat{F}_1)
+(\,\overast{D}_{\widehat{F}_j}\,\overast{R})(\widehat{F}_1,\widehat{F}_i)=0.$$
Since $\widehat{F}_1\in \N_{R^\ast}$ and\, $\overast{T}(hX,hY)=\mathfrak{R}(X,Y)=0$, the last equality  takes the form
$$\,\overast{D}_{\widehat{F}_1}\,\overast{R}(\widehat{F}_i,\widehat{F}_j)+\,\overast{R}(\widehat{F}_j,[\widehat{F}_1,\widehat{F}_i])
+\,\,\overast{R}(\widehat{F}_i,[\widehat{F}_j,\widehat{F}_1])=0.$$
Applying the above equation on $\widehat{F}_a$, we get
\begin{equation}\label{drrr}
\,\overast{D}_{\widehat{F}_1}\,\overast{R}(\widehat{F}_i,\widehat{F}_j)\widehat{F}_a
+\,\overast{R}(\widehat{F}_j,[\widehat{F}_1,\widehat{F}_i])\widehat{F}_a+\,\,\overast{R}(\widehat{F}_i,[\widehat{F}_j,\widehat{F}_1])\widehat{F}_a=0.
\end{equation}
Since, $[\widehat{F}_1,\widehat{F}_i]$ is horizontal, it can be written in the form $[\widehat{F}_1,\widehat{F}_i]=\xi^k_{1i}\widehat{F}_k+\xi^\mu_{1i}\widehat{F}_\mu$, where $k=r_0+1,...,n$ and $\mu=1,...,r_0$. Consequently, by (\ref{complete}) and (\ref{drrr}), noting that $\widehat{F}_\mu$ are null vector fields, we get
\begin{equation}\label{dash}
  (\phi^h_{ija})'+\xi^k_{1i}\,\phi^h_{jka}-\xi^k_{1j}\,\phi^h_{ika}=0
\end{equation}
Since $F_a$ is  a nullity vector field at $p$, then for the fixed index $a$, $\phi^h_{lma}(p) = 0$, where $l,m = r_0+1, ..., n$. Hence, the differential equations (\ref{dash}) with the  initial condition $\phi^h_{lma}(p) = 0 $  imply that the functions $\phi^h_{lma}$ vanish identically.  As $\mathfrak{R}=0$, Theorem \ref{coincidenulker} and (\ref{complete}) give rise to
\begin{equation}\label{final}
  F_a(\gamma(t))\in \mathcal{N}_{R^\ast}(\gamma(t)),\,\, \text{for all}\,\, t\in [0,c]
\end{equation}
Now (\ref{contradiction-1}) and (\ref{final}) lead to a contradiction.
Consequently,  $\gamma$  can be extended to a geodesic $\widetilde{\gamma} : [0,\infty) \longrightarrow N$.
\end{proof}

\Section{Nullity distribution of the Chern hv-curvature}
In this section we  investigate the nullity distribution of the hv-curvature \, $\overast{P}$ of the Chern connection. We show, by a counterexample,  that the nullity  distribution $\N_{P^\ast}$  is not   completely integrable. We find a sufficient condition for $\N_{P^\ast}$ to be completely integrable.

\begin{defn} Let \, $\overast{P}$ be the hv-curvature of the  Chern connection.
The  nullity space  of \, $\overast{P}$ at a point $z\in TM$ is a subspace of $H_z(TM)$ defined by
$$\mathcal{N}_{P^\ast}(z):=\{v\in H_z(TM) | \,\,  \overast{P}_z(v,w)=0, \, \,\text{for all}\,\, w\in H_z(TM)\}.$$
The dimension of $\mathcal{N}_{P^\ast}(z)$, denoted by $\mu_{P^\ast}(z)$, is the nullity index   of \, $\overast{P}$ at $z$.
\end{defn}
\begin{prop}\label{pchern.nullity}The  nullity distribution  of \, $\overast{P}$ satisfies:
\begin{enumerate}
\item[\textup{(1)}]$ \N_{P^\ast}\neq\phi$.

 \item[\textup{(2)}] If $X\in \Gamma(\N_{P^\ast})$, then $[C,X]\in \Gamma(\N_{P^\ast})$.

  \item[\textup{(3)}] If $X\in \Gamma(\N_{P^\ast})$, then $\mathcal{C}'(X,Y)=0,\, \,\text{for all}\,\, Y\in \mathfrak{X}^h(\T M).$

  \item[\textup{(4)}] If $\mu_{P^\ast}=n$, then\, $\N_{{R^\ast}}=\N_{R^\circ}$,

\end{enumerate}
where $\N_{R^\circ}$ is the nullity distribution of the h-curvature of the Berwald connection \cite{Nabil.1}.
\end{prop}

A Finsler manifold is said to be Landsbergian if the Landsberg tensor $\mathcal{C}'$ vanishes or, equivalently, $P=0$ \cite{szilasi}.
 If  the  nullity index    $\mu_{P^\ast}$ takes its maximum, then by Proposition \ref{pchern.nullity} (3),   $\mathcal{C}'=0.$
  Consequently, a Finsler manifold $(M,E)$ is Landsbergian  if  the  nullity index $\mu_{P^\ast}$  achieves  its maximum.

\begin{thm}
A Finsler manifold $(M,E)$ is Landsbergian  if and only if the canonical spray $S$  is a nullity vector field for the  the  distribution  $\N_{P^\ast}$.
\end{thm}

\begin{proof}
By  (\ref{eq.4}),   we have
\begin{eqnarray*}
   (M,E)\,\text{is Landsbergian}&\Longleftrightarrow&\mathcal{C}'=0 \\
     &\Longleftrightarrow& \overast{P}(X,Y)S=0\,\, \,\text{for all}\,\, X,Y \in \cppp \\
  &\Longleftrightarrow&  \overast{P}(S,Y)X=0 \,\,\,\text{for all}\,\, X,Y \in \cppp\\
 &\Longleftrightarrow& S\in \Gamma(\N_{P^\ast}),
\end{eqnarray*}
as was to be shown.
\end{proof}

\begin{rem}
The above theorem shows that the canonical spray $S$ does not belong to the nullity distribution $\N_{P^\ast}$ except in the  Landsbergian case.
 This is in contrast to the case of Cartan connection, where  the canonical spray always belongs to the nullity  distribution of the Cartan  hv-curvature $P$.
\end{rem}

The nullity distribution $\N_{P^\ast}$ is not  completely integrable in general, as is illustrated  by the following  example.

\bigskip

\begin{example}
Let
$U=\{(x^1,x^2,x^3;y^1,y^2,y^3)\in\mathbb{R}^3\times \mathbb{R}^3:\,y^1, y^2,y^3\neq 0, y^3\neq 4y^2 \}\subset TM$, where $M:= \mathbb{R}^3$. Define  $F$ on $U$ by
$$F(x,y):=\sqrt [4]{{{ e}^{-{ x^1 x^2}}}{{ (y^1)}}^{2}{{ (y^3)}}^{2}{{e}^{-{\frac {{ y^3}}{{ y^2}}}}}}.$$
By Maple program and NF-package we can perform the following calculations.
\bigskip

\begin{maplegroup}
\begin{mapleinput}
\mapleinline{active}{1d}{F0 := (exp(-x1x2)*y1\symbol{94}2*y3\symbol{94}2*exp(-y3/(y2)))\symbol{94}(1/2);
}{}
\end{mapleinput}
\mapleresult
\begin{maplelatex}
\mapleinline{inert}{2d}{F0 := sqrt(x2^2*y1^4+y2^4+y3^4+y4^4)}{\[ F0\, := \,\sqrt{{{\rm e}^{-{\it x1 x2}}}{{\it y1}}^{2}{{\it y3}}^{2}{{\rm e}^{-{\frac {{\it y3}}{{\it y2}}}}}}\]}
\end{maplelatex}
\end{maplegroup}

\begin{maplegroup}
\begin{Maple Normal}{
{{\bf Barthel connection}}}\end{Maple Normal}

\end{maplegroup}
\begin{maplegroup}
\begin{mapleinput}
\mapleinline{active}{1d}{show(N[i,-j]);
}{}
\end{mapleinput}
\mapleresult
\begin{maplelatex}
\mapleinline{inert}{2d}{}{}{\[ N^{{\it x1} }_{{\it x1} }=-\frac{1}{2} {\it x2 \it y1}\hspace{1cm} N^{{\it x2} }_{{\it x2} }=-\frac{4{\it x1} {\it y2}^{3}\left(3{\it y2} -{\it y3}\right)}{\left(-{\it y3}+4{\it y2} \right)^{2}{\it y3}}\hspace{1cm} N^{{\it x2} }_{{\it x3} }=\frac{2 {\it x1}{\it y2}^{4}\left(2{\it y2}-{\it y3} \right)}{\left(-{\it y3} +4{\it y2}\right)^{2}{\it y3}^{2}}\]}
\end{maplelatex}
\mapleresult
\begin{maplelatex}
\mapleinline{inert}{2d}{}{}{\[ N^{{\it x3} }_{{\it x2} }=-\frac{{\it x1}{\it y3} \left(2{\it y2}-{\it y3} \right) {\it y2}}{\left(-{\it y3} +4{\it y2}\right)^{2}}\hspace{1cm} N^{{\it x3} }_{{\it x3} }=-\frac{2 {\it x1}{\it y2}^{3}}{\left(-{\it y3}+4{\it y2} \right)^{2}}\]}
\end{maplelatex}
\end{maplegroup}

\begin{maplegroup}
\begin{Maple Normal}{
{\bf Chern hv-curvature\, $\overast{P}$:}}\end{Maple Normal}

\end{maplegroup}
\begin{maplegroup}
\begin{mapleinput}
\mapleinline{active}{1d}{definetensor(Pchern[i,-h,-j,-k] = tddiff(Gammastar[i,-h,-j],
 Y[k]),symm[2,3]); }{\[\]}
\end{mapleinput}
\mapleresult
\begin{mapleinput}
\mapleinline{active}{1d}{show(Pchern[h, -i, -j, -k]); }{\[\]}
\end{mapleinput}
\mapleresult
\begin{maplelatex}
\mapleinline{inert}{2d}{}{\[ {\it Pchern}^{{\it x2} }_{{\it x2 x2 x2} }=-\frac{12 x1{\it y2}\left(-{\it y3}^{3}+8{\it y3}^{2}{\it y2}-24{\it y2}^{2}{\it y3} +24{\it y2}^{3}\right)}{{\it y3} \left(-{\it y3} +4{\it y2}\right)^{4}}\]}
\end{maplelatex}
\mapleresult
\begin{maplelatex}
\mapleinline{inert}{2d}{}{\[ {\it Pchern}^{{\it x2} }_{{\it x2 x2 x3} }=\frac{12x1{\it y2}^{2}\left(-{\it y3}^{3}+8{\it y3}^{2}{\it y2} -24{\it y2}^{2}{\it y3} +24{\it y2}^{3}\right)}{{\it y3}^{2}\left(-{\it y3} +4{\it y2} \right)^{4}}\]}
\end{maplelatex}
\mapleresult
\begin{maplelatex}
\mapleinline{inert}{2d}{}{\[ {\it Pchern}^{{\it x3} }_{{\it x2 x2 x2} }=\frac{6x1{\it y3} \left({\it y3}^{2}-4{\it y2} {\it y3} +8{\it y2}^{2}\right)}{\left(-{\it y3} +4{\it y2} \right)^{4}}\]}
\end{maplelatex}
\mapleresult
\begin{maplelatex}
\mapleinline{inert}{2d}{}{\[ {\it Pchern}^{{\it x3} }_{{\it x2 x2x3} }=-\frac{6x1{\it y2} \left({\it y3}^{2}-4{\it y2} {\it y3} +8{\it y2}^{2}\right)}{\left(-{\it y3} +4{\it y2} \right)^{4}}\]}
\end{maplelatex}
\mapleresult
\begin{maplelatex}
\mapleinline{inert}{2d}{}{\[ {\it Pchern}^{{\it x2} }_{{\it x2 x3 x2} }=\frac{6x1{\it y2}^{2}\left(-28{\it y2}^{2}{\it y3} +32{\it y2}^{3}+8{\it y3}^{2}{\it y2} -{\it y3}^{3}\right)}{{\it y3}^{2}\left(-{\it y3} +4{\it y2} \right)^{4}}\]}
\end{maplelatex}
\mapleresult
\begin{maplelatex}
\mapleinline{inert}{2d}{}{\[ {\it Pchern}^{{\it x2} }_{{\it x2 x3 x3} }=-\frac{6x1{\it y2}^{3}\left(-28{\it y2}^{2}{\it y3} +32{\it y2}^{3}+8{\it y3}^{2}{\it y2} -{\it y3}^{3}\right)}{{\it y3}^{3}\left(-{\it y3} +4{\it y2} \right)^{4}}\]}
\end{maplelatex}
\mapleresult
\begin{maplelatex}
\mapleinline{inert}{2d}{}{}{\[ {\it Pchern}^{{\it x3} }_{{\it x2 x3 x2} }=-\frac{12x1{\it y2}^{2}{\it y3} }{\left(-{\it y3} +4{\it y2} \right)^{4}}\hspace{2.6cm} {\it Pchern}^{{\it x3} }_{{\it x2 x3 x3} }=\frac{12x1{\it y2}^{3}}{\left(-{\it y3} +4{\it y2} \right)^{4}}\]}
\end{maplelatex}
\mapleresult
\begin{maplelatex}
\mapleinline{inert}{2d}{}{}{\[ {\it Pchern}^{{\it x2} }_{{\it x3 x3 x2} }=-\frac{48x1{\it y2}^{5}\left(2{\it y2} -{\it y3} \right)}{{\it y3}^{3}\left(-{\it y3} +4{\it y2} \right)^{4}}\hspace{2cm} {\it Pchern}^{{\it x2} }_{{\it x3 x3 x3} }=\frac{48x1{\it y2}^{6}\left(2{\it y2} -{\it y3} \right)}{{\it y3}^{4}\left(-{\it y3} +4{\it y2} \right)^{4}}\]}
\end{maplelatex}
\mapleresult
\begin{maplelatex}
\mapleinline{inert}{2d}{}{}{\[ {\it Pchern}^{{\it x3} }_{{\it x3 x3 x2} }=-\frac{6x1{\it y2}^{2}\left(-8{\it y2} {\it y3} +8{\it y2}^{2}+{\it y3}^{2}\right)}{{\it y3} \left(-{\it y3} +4{\it y2} \right)^{4}}\hspace{.6cm}{\it Pchern}^{{\it x3} }_{{\it x3 x3 x3} }=\frac{6 x1{\it y2}^{3}\left(-8{\it y2}{\it y3} +8{\it y2}^{2}+{\it y3}^{2}\right)}{{\it y3}^{2}\left(-{\it y3} +4{\it y2}\right)^{4}}\]}
\end{maplelatex}
\end{maplegroup}
\begin{maplegroup}
\begin{Maple Normal}{
{\bf  \, $\overast{P}$-nullity vectors:}}\end{Maple Normal}

\end{maplegroup}
\begin{maplegroup}
\begin{mapleinput}
\mapleinline{active}{1d}{definetensor(PchernW[h, -i, -k] = Pchern[h, -i, -j, -k]*w[j]); }{\[\]}
\end{mapleinput}
\end{maplegroup}
\begin{maplegroup}
\begin{mapleinput}
\mapleinline{active}{1d}{show(PchernW[h, -i, -k]); }{\[\]}
\end{mapleinput}
\mapleresult
\begin{maplelatex}
\mapleinline{inert}{2d}{}{\[ {\it PchernW}^{{\it x2} }_{{\it x2  x2} }=-\frac{12x1 {\it y2}\left(8{\it y3}^{2}{\it y2}-{\it y3}^{3}-24{\it y2}^{2}{\it y3} +24{\it y2}^{3}\right)w^{{\it x2} }}{{\it y3} \left(-{\it y3} +4{\it y2}\right)^{4}}+\frac{6x1 {\it y2}^{2}\left(32{\it y2}^{3}+8{\it y3}^{2}{\it y2}-28{\it y2}^{2}{\it y3} -{\it y3}^{3}\right)w^{{\it x3} }}{{\it y3}^{2}\left(-{\it y3} +4{\it y2}\right)^{4}}\]}
\end{maplelatex}
\mapleresult
\begin{maplelatex}
\mapleinline{inert}{2d}{}{\[ {\it PchernW}^{{\it x2} }_{{\it x2  x3} }=\frac{12x1{\it y2}^{2}\left(8{\it y3}^{2}{\it y2}-{\it y3}^{3} -24{\it y2}^{2}{\it y3} +24{\it y2}^{3}\right)w^{{\it x2} }}{{\it y3}^{2}\left(-{\it y3} +4{\it y2} \right)^{4}}-\frac{6x1{\it y2}^{3}\left(32{\it y2}^{3}-28{\it y2}^{2}{\it y3}+8{\it y3}^{2}{\it y2} -{\it y3}^{3}\right)w^{{\it x3} }}{{\it y3}^{3}\left(-{\it y3} +4{\it y2} \right)^{4}}\]}
\end{maplelatex}
\mapleresult
\begin{maplelatex}
\mapleinline{inert}{2d}{}{\[ {\it PchernW}^{{\it x2} }_{{\it x3  x2} }=\frac{6x1{\it y2}^{2}\left(-28{\it y2}^{2}{\it y3} +32{\it y2}^{3}+8{\it y3}^{2}{\it y2} -{\it y3}^{3}\right)w^{{\it x2} }}{{\it y3}^{2}\left(-{\it y3} +4{\it y2} \right)^{4}}-\frac{48x1{\it y2}^{5}\left(2{\it y2} -{\it y3} \right)w^{{\it x3} }}{{\it y3}^{3}\left(-{\it y3} +4{\it y2} \right)^{4}}\]}
\end{maplelatex}
\mapleresult
\begin{maplelatex}
\mapleinline{inert}{2d}{}{\[ {\it PchernW}^{{\it x2} }_{{\it x3  x3} }=-\frac{6{\it y2}^{3}\left(-28x1{\it y2}^{2}{\it y3} +32{\it y2}^{3}+8{\it y3}^{2}{\it y2} -{\it y3}^{3}\right)w^{{\it x2} }}{{\it y3}^{3}\left(-{\it y3} +4{\it y2} \right)^{4}}+\frac{48x1{\it y2}^{6}\left(2{\it y2} -{\it y3} \right)w^{{\it x3} }}{{\it y3}^{4}\left(-{\it y3} +4{\it y2} \right)^{4}}\]}
\end{maplelatex}
\mapleresult
\begin{maplelatex}
\mapleinline{inert}{2d}{}{\[ {\it PchernW}^{{\it x3} }_{{\it x2  x2} }=\frac{6x1{\it y3} \left({\it y3}^{2}-4{\it y2} {\it y3} +8{\it y2}^{2}\right)w^{{\it x2} }}{\left(-{\it y3} +4{\it y2} \right)^{4}}-\frac{12x1{\it y2}^{2}{\it y3} w^{{\it x3} }}{\left(-{\it y3} +4{\it y2} \right)^{4}}\]}
\end{maplelatex}
\mapleresult
\begin{maplelatex}
\mapleinline{inert}{2d}{}{\[ {\it PchernW}^{{\it x3} }_{{\it x2  x3} }=-\frac{6x1{\it y2} \left({\it y3}^{2}-4{\it y2} {\it y3} +8{\it y2}^{2}\right)w^{{\it x2} }}{\left(-{\it y3} +4{\it y2} \right)^{4}}+\frac{12x1{\it y2}^{3}w^{{\it x3} }}{\left(-{\it y3} +4{\it y2} \right)^{4}}\]}
\end{maplelatex}
\mapleresult
\begin{maplelatex}
\mapleinline{inert}{2d}{}{\[ {\it PchernW}^{{\it x3} }_{{\it x3  x2} }=-\frac{12x1{\it y2}^{2}{\it y3} w^{{\it x2} }}{\left(-{\it y3} +4{\it y2} \right)^{4}}-\frac{6x1{\it y2}^{2}\left(-8{\it y2} {\it y3} +8{\it y2}^{2}+{\it y3}^{2}\right)w^{{\it x3} }}{{\it y3} \left(-{\it y3} +4{\it y2} \right)^{4}}\]}
\end{maplelatex}
\mapleresult
\begin{maplelatex}
\mapleinline{inert}{2d}{}{\[ {\it PchernW}^{{\it x3} }_{{\it x3  x3} }=\frac{12x1 {\it y2}^{3}w^{{\it x2} }}{\left(-{\it y3} +4{\it y2}\right)^{4}}+\frac{6x1 {\it y2}^{3}\left(-8{\it y2}{\it y3} +8{\it y2}^{2}+{\it y3}^{2}\right)w^{{\it x3} }}{{\it y3}^{2}\left(-{\it y3} +4{\it y2}\right)^{4}}\]}
\end{maplelatex}
\end{maplegroup}

\bigskip

Putting ${\it PchernW}^{\it h }_{\it ij }=0$,  we get a system of algebraic equations. This system has a  solution if  $y_3=2y_2$ and $x^1>0$: $W^1=s$, $W^2=t$, $W^3=2t$, $
 \, s, t\in \mathbb{R}$.
Hence, a \,$\overast{P}$-nullity vector must have the form  $W=sh_1+t(h_2+2h_3)$, where the horizontal basis vector fields $h_1,h_2,h_3$ are given by $h_1=\frac{\partial}{\partial x_1}+\frac{x_2y_1}{2}\frac{\partial}{\partial y_1}$, $h_2=\frac{\partial}{\partial x_2}+\frac{x_1y_2}{2}\frac{\partial}{\partial y_2}$, $h_3=\frac{\partial}{\partial x_3}+\frac{x_1y_2}{2}\frac{\partial}{\partial y_3}$. Now, take $X,Y\in\N_{P^\ast}$ such that  $X=h_1$, $Y=h_2+2h_3$.
   Hence, the bracket $[X,Y]=[h_1,h_2+2h_3]= -\frac{y_1}{2}\frac{\partial}{\partial y_1}+\frac{y_2}{2}\frac{\partial}{\partial y_2}+\frac{y_2}{2}\frac{\partial}{\partial y_3}$   is vertical and,  consequently, $\N_{P^\ast}$ is not completely integrable.
\end{example}

\begin{thm}\label{p.is.c.i.}
Let  $\mu_{P^\ast}$ be constant on an open subset $U$ of $TM$.  The nullity distribution $\N_{P^\ast}$ is completely integrable on $U$ if and only if\,  $ \mathfrak{R}(X,Y)=0 \,\, \text{and}\,\, (\,\overast{D}_{JZ}\,\,\overast{R})(X,Y)=~0$, for all $ X,Y \in \Gamma(\N_{P^\ast})$.
\end{thm}
\begin{proof}
\emph{Necessity}. Let $\N_{P^\ast}$ be completely integrable. Then, if $X,Y\in \Gamma(\N_{P^\ast})$,  the bracket $[hX,hY]$ is horizontal, thus, $\mathfrak{R}(X,Y)=0$. Also,  by (\ref{eq.8}) and the fact that\, $\overast{P}([hX,hY],Z)=( \,\overast{D}_{hX}\,\overast{P})(Y,Z)-(\,\overast{D}_{hY}\,\overast{P})(X,Z)=0$ (by (\ref{chern.[]})),  we have \\$ (\,\overast{D}_{JZ}\,\overast{R})(X,Y)=0,\, \text{for all}\,\, X,Y \in \Gamma(\N_{P^\ast}), \, \text{for all}\,\, Z\in \cppp$.

\emph{ Sufficiency}. Let  $\mathfrak{R}(X,Y)=0$  and $ (\,\overast{D}_{JZ}\,\,\overast{R})(X,Y)=0$ for all $X, Y\in \Gamma(\N_{P^\ast})$. As  $0=\mathfrak{R}(X,Y)=-v[hX,hY]=-v[X,Y]$,  the bracket
 $[X,Y]$ is horizontal. Making use of (\ref{chern.[]}) and (\ref{eq.8}), we get
 \begin{eqnarray*}
  ( \,\overast{D}_{hX}\,\overast{P})(Y,Z)-(\,\overast{D}_{hY}\,\overast{P})(X,Z)=0&\Longrightarrow&\overast{P}(\,\overast{D}_{X}Y-\,\overast{D}_{Y}X,Z)=0 \\
    &\Longrightarrow& \overast{P}([X,Y]+\mathfrak{R}(X,Y),Z)=0 \\
    &\Longrightarrow& \overast{P}([X,Y],Z)=0\\
    &\Longrightarrow&[X,Y]\in \Gamma(\N_{P^\ast}).
 \end{eqnarray*}
Hence $\N_{P^\ast}$ is completely integrable.
\end{proof}

By the property  \,  $\overast{P}(X,Y)Z=\overast{P}(Z,Y)X$ we have the following result.

\begin{thm}
The nullity distribution  $\N_{P^\ast}$ and the kernel distribution  $\text{Ker}_{P^\ast}$   coincide.
\end{thm}

 A Finsler manifold in which   the Chern hv-curvature tensor \, $\overast{P}$   vanishes is called a Berwald space \cite{szilasi}. It is well known that every Berwald space is a Landsberg space, but it is not known whether the converse is   true. In \cite{Shen-Landsberg}, Shen introduced a class of non-regular Finsler metrics which is Landsbergian and not Berwaldian. The calculations are not easy, especially, if one wants to study some concrete examples.   Here, by using Maple program together with the results of  \cite{Shen-Landsberg} and \cite{CFG}, we give a simple  class of  proper non-regular non Berwaldian Landsbergian spaces.\\

\begin{example}\label{ex.3} Let
$M= \mathbb{R}^3$,
$U=\{(x^1,x^2,x^3;y^1,y^2,y^3)\in\mathbb{R}^3\times \mathbb{R}^3:\, y^2>0,y^3 >0  \}\subset TM$. Define   $F$  on $U$ by
$$F(x,y):=f(x^1)\sqrt{{(y^1)}^2+y^2y^3+y^1\sqrt{y^2y^3}}\,e^{\frac{1}{\sqrt{3}}\arctan\Big(\frac{2y^1}{\sqrt{3y^2y^3}}+\frac{1}{\sqrt{3}}\Big)}.$$
The idea is to compute the Landsberg tensor $L_{ijk}$ and the Berwald tensor $G^h_{ijk}$ which are locally given  by
$$L_{ijk}:=\frac{F}{2}\frac{\partial F}{\partial y^h}G^h_{ijk},\,\,\,\,G^h_{ijk}:=\frac{\partial^3 G^h}{\partial y^i\partial y^j\partial y^k}.$$
Then, we show that the Landsberg tensor vanishes identically while there are some non vanishing components of the Berwald tensor (for simplicity we consider only one nonzero component and check it at a point) .
\begin{Maple Normal}{
\begin{Maple Normal}{
\mapleinline{inert}{2d}{}{\[\displaystyle \]}
}\end{Maple Normal}
}\end{Maple Normal}
\begin{maplegroup}
\begin{mapleinput}
\mapleinline{active}{1d}{restart}{\[{\it restart}\]}
\end{mapleinput}
\end{maplegroup}
\begin{maplegroup}
\begin{mapleinput}
\mapleinline{active}{1d}{}{\[F\, := \,f \left( {\it x1} \right)  \sqrt{{{\it y1}}^{2}+{\it y2}\,{\it y3}+{\it y1}\, \sqrt{{\it y2}\,{\it y3}}}\\
\mbox{}{{\rm e}^{\arctan \left( 2\,{\frac {{\it y1}}{ \sqrt{3} \sqrt{{\it y2}\,{\it y3}}}}+ \left(  \sqrt{3} \right) ^{-1} \right)  \left(  \sqrt{3} \right) ^{-1}}}\]}
\end{mapleinput}
\mapleresult
\begin{maplelatex}
\mapleinline{inert}{2d}{}{\[\displaystyle F\, := \,f \left( {\it x1} \right)  \sqrt{{{\it y1}}^{2}+{\it y2}\,{\it y3}+{\it y1}\, \sqrt{{\it y2}\,{\it y3}}}\,\,{{\rm e}^{\frac{\sqrt{3}}{3}\,\arctan \left( \frac{2}{3}\,{\frac {{\it y1}\, \sqrt{3}}{ \sqrt{{\it y2}\,{\it y3}}}}+\frac{\sqrt{3}}{3} \right)  }}\]}
\end{maplelatex}
\end{maplegroup}
\begin{maplegroup}
\begin{mapleinput}
\mapleinline{active}{2d}{}{\[\]}
\end{mapleinput}
\end{maplegroup}
\begin{maplegroup}
\begin{mapleinput}
\mapleinline{active}{2d}{}{\[\]}
\end{mapleinput}
\mapleresult
\begin{maplegroup}
\begin{mapleinput}
\mapleinline{active}{1d}{simplify(G1)}{\[\]}
\end{mapleinput}
\end{maplegroup}
\begin{maplelatex}
\mapleinline{inert}{2d}{}{\[\displaystyle {\it G1}\, := \frac{1}{2}\,{\frac { \left( {{\it y1}}^{2}-{\it y2}\,{\it y3} \right) {\frac {d}{d{\it x1}}}f \left( {\it x1} \right) }{f \left( {\it x1} \right) }}\]}
\end{maplelatex}
\end{maplegroup}
\begin{maplegroup}
\begin{maplegroup}
\begin{mapleinput}
\mapleinline{active}{1d}{simplify(G2)}{\[\]}
\end{mapleinput}
\end{maplegroup}
\mapleresult
\begin{maplelatex}
\mapleinline{inert}{2d}{}{}{\[\displaystyle {\it G2}\, := \frac{1}{2}\,  \left( {\frac {d}{d{\it x1}}}f \left( {\it x1} \right)  \right) {{\it y2}}^{2}{\it y3}\, ( 92\,{{\it y2}}^{5}{{\it y3}}^{5}{{\it y1}}^{3}+408\,{{\it y2}}^{3}{{\it y3}}^{3}{{\it y1}}^{7}+230\,{{\it y2}}^{2}{{\it y3}}^{2}{{\it y1}}^{9}\]}
\end{maplelatex}
\begin{maplelatex}
\mapleinline{inert}{2d}{}{}{\[
+48\,{\it y2}\,{\it y3}\,{{\it y1}}^{11}
+8\,{{\it y2}}^{6}{{\it y3}}^{6}{\it y1}+306\,{{\it y2}}^{4}{{\it y3}}^{4}{{\it y1}}^{5}
+2\,{{\it y1}}^{13}+({{\it y2}}^{6}{{\it y3}}^{6} +33\,{{\it y2}}^{5}{{\it y3}}^{5} {{\it y1}}^{2}\]}
\end{maplelatex}
\begin{maplelatex}
\mapleinline{inert}{2d}{}{}{\[
+190\,{{\it y2}}^{4}{{\it y3}}^{4} {{\it y1}}^{4}
+121\,{\it y2}\,{\it y3}\, {{\it y1}}^{10}+342\,{{\it y2}}^{2}{{\it y3}}^{2} {{\it y1}}^{8}
+393\,{{\it y2}}^{3}{{\it y3}}^{3} {{\it y1}}^{6}\]}
\end{maplelatex}
\begin{maplelatex}
\mapleinline{inert}{2d}{}{}{\[
+13\, {{\it y1}}^{12})\sqrt{{\it y2}\,{\it y3}} ) /(f \left( {\it x1} \right)  ( 50\,{{\it y2}}^{5}{{\it y3}}^{5}{{\it y1}}^{3}+126\,{{\it y2}}^{3}{{\it y3}}^{3}{{\it y1}}^{7}+50\,{{\it y2}}^{2}{{\it y3}}^{2}{{\it y1}}^{9}\]}
\end{maplelatex}
\begin{maplelatex}
\mapleinline{inert}{2d}{}{}{\[
+6\,{\it y2}\,{\it y3}\,{{\it y1}}^{11}+6\,{{\it y2}}^{6}{{\it y3}}^{6}{\it y1}+126\,{{\it y2}}^{4}{{\it y3}}^{4}{{\it y1}}^{5}+({{\it y2}}^{6}{{\it y3}}^{6}
+21\,{{\it y2}}^{5}{{\it y3}}^{5} {{\it y1}}^{2}+ {{\it y1}}^{12}\]}
\end{maplelatex}
\begin{maplelatex}
\mapleinline{inert}{2d}{}{}{\[
+90\,{{\it y2}}^{4}{{\it y3}}^{4} {{\it y1}}^{4}+21\,{\it y2}\,{\it y3}\, {{\it y1}}^{10}+90\,{{\it y2}}^{2}{{\it y3}}^{2} {{\it y1}}^{8}+141\,{{\it y2}}^{3}{{\it y3}}^{3} {{\it y1}}^{6}
)\sqrt{{\it y2}\,{\it y3}} )  \sqrt{{\it y2}\,{\it y3}})\]}
\end{maplelatex}
\end{maplegroup}
\vspace{8pt}
\begin{maplegroup}
\begin{maplegroup}
\begin{mapleinput}
\mapleinline{active}{1d}{simplify(G3)}{\[\]}
\end{mapleinput}
\end{maplegroup}
\mapleresult
\begin{maplelatex}
\mapleinline{inert}{2d}{}{}{\[\displaystyle {\it G3}\, := \frac{1}{2}\, \left( {\frac {d}{d{\it x1}}}f \left( {\it x1} \right)  \right) {{\it y3}}^{2}{\it y2}\, ( 408\,{{\it y2}}^{3}{{\it y3}}^{3}{{\it y1}}^{7}+230\,{{\it y2}}^{2}{{\it y3}}^{2}{{\it y1}}^{9}+8\,{{\it y2}}^{6}{{\it y3}}^{6}{\it y1}\]}
\end{maplelatex}
\begin{maplelatex}
\mapleinline{inert}{2d}{}{}{\[
+2\,{{\it y1}}^{13}+(33\,{{\it y1}}^{2}{{\it y2}}^{5}{{\it y3}}^{5} +393\,{{\it y1}}^{6}{{\it y2}}^{3}{{\it y3}}^{3}+342\,{{\it y1}}^{8}{{\it y3}}^{2}{{\it y2}}^{2}+121\,{{\it y1}}^{10}{\it y3}\,{\it y2}\,  \]}
\end{maplelatex}
\begin{maplelatex}
\mapleinline{inert}{2d}{}{}{\[
 +190\,{{\it y1}}^{4}{{\it y3}}^{4}{{\it y2}}^{4} +13\,{{\it y1}}^{12} +{{\it y2}}^{6}{{\it y3}}^{6})\sqrt{{\it y2}\,{\it y3}}+92\,{{\it y2}}^{5}{{\it y3}}^{5}{{\it y1}}^{3} \]}
\end{maplelatex}
\begin{maplelatex}
\mapleinline{inert}{2d}{}{}{\[
+306\,{{\it y2}}^{4}{{\it y3}}^{4}{{\it y1}}^{5}+48\,{\it y2}\,{\it y3}\,{{\it y1}}^{11} ) /(f \left( {\it x1} \right)  ( 50\,{{\it y2}}^{2}{{\it y3}}^{2}{{\it y1}}^{9}+6\,{\it y2}\,{\it y3}\,{{\it y1}}^{11}\]}
\end{maplelatex}
\begin{maplelatex}
\mapleinline{inert}{2d}{}{}{\[
+126\,{{\it y2}}^{3}{{\it y3}}^{3}{{\it y1}}^{7}+6\,{{\it y2}}^{6}{{\it y3}}^{6}{\it y1}+(90\,{{\it y1}}^{4}{{\it y3}}^{4}{{\it y2}}^{4} +141\,{{\it y1}}^{6}{{\it y2}}^{3}{{\it y3}}^{3} \sqrt{{\it y2}\,{\it y3}}\]}
\end{maplelatex}
\begin{maplelatex}
\mapleinline{inert}{2d}{}{}{\[
+21\,{{\it y1}}^{10}{\it y3}\,{\it y2}\, +90\,{{\it y1}}^{8}{{\it y3}}^{2}{{\it y2}}^{2} +21\,{{\it y1}}^{2}{{\it y2}}^{5}{{\it y3}}^{5} +{{\it y1}}^{12} +{{\it y2}}^{6}{{\it y3}}^{6})\sqrt{{\it y2}\,{\it y3}}\]}
\end{maplelatex}
\begin{maplelatex}
\mapleinline{inert}{2d}{}{}{\[
 +126\,{{\it y2}}^{4}{{\it y3}}^{4}{{\it y1}}^{5}+50\,{{\it y2}}^{5}{{\it y3}}^{5}{{\it y1}}^{3} )  \sqrt{{\it y2}\,{\it y3}})\]}
\end{maplelatex}
\end{maplegroup}

\begin{maplegroup}
\begin{mapleinput}
\mapleinline{active}{2d}{}{\[\]}
\end{mapleinput}
\end{maplegroup}
\begin{maplegroup}
\begin{mapleinput}
\mapleinline{active}{1d}{y1 := y[1]; 1; y2 := y[2]; 1; y3 := y[3]}{\[\]}
\end{mapleinput}
\mapleresult
\begin{maplelatex}
\mapleinline{inert}{2d}{y1 := y[1]}{\[\displaystyle {\it y1}\, := \,{\it y}_{{1}}\]}
\end{maplelatex}
\mapleresult
\begin{maplelatex}
\mapleinline{inert}{2d}{y2 := y[2]}{\[\displaystyle {\it y2}\, := \,{\it y}_{{2}}\]}
\end{maplelatex}
\mapleresult
\begin{maplelatex}
\mapleinline{inert}{2d}{y3 := y[3]}{\[\displaystyle {\it y3}\, := \,{\it y}_{{3}}\]}
\end{maplelatex}
\end{maplegroup}
\vspace{8pt}
\begin{maplegroup}
\begin{mapleinput}
\mapleinline{active}{1d}{printlevel := 3;
 for i to 3 do
  for j to i do
   for k to j do
  Landsberg[i,j,k] := simplify((diff(F,y1))*(diff(G1,y[i],y[j],y[k]))
   +(diff(F,y2))*(diff(G2,y[i],y[j],y[k]))
   +(diff(F,y3))*(diff(G3,y[i],y[j],y[k])));
    end do;
   end do;
   end do; }{\[\]}
\end{mapleinput}
\mapleresult
\begin{maplelatex}
\mapleinline{inert}{2d}{Landsberg[1, 1, 1] := 0}{\[\displaystyle {\it Landsberg}_{{1,1,1}}\, := \,0\]}
\end{maplelatex}
\mapleresult
\begin{maplelatex}
\mapleinline{inert}{2d}{Landsberg[2, 1, 1] := 0}{\[\displaystyle {\it Landsberg}_{{2,1,1}}\, := \,0\]}
\end{maplelatex}
\mapleresult
\begin{maplelatex}
\mapleinline{inert}{2d}{Landsberg[2, 2, 1] := 0}{\[\displaystyle {\it Landsberg}_{{2,2,1}}\, := \,0\]}
\end{maplelatex}
\mapleresult
\begin{maplelatex}
\mapleinline{inert}{2d}{Landsberg[2, 2, 2] := 0}{\[\displaystyle {\it Landsberg}_{{2,2,2}}\, := \,0\]}
\end{maplelatex}
\mapleresult
\begin{maplelatex}
\mapleinline{inert}{2d}{Landsberg[3, 1, 1] := 0}{\[\displaystyle {\it Landsberg}_{{3,1,1}}\, := \,0\]}
\end{maplelatex}
\mapleresult
\begin{maplelatex}
\mapleinline{inert}{2d}{Landsberg[3, 2, 1] := 0}{\[\displaystyle {\it Landsberg}_{{3,2,1}}\, := \,0\]}
\end{maplelatex}
\mapleresult
\begin{maplelatex}
\mapleinline{inert}{2d}{Landsberg[3, 2, 2] := 0}{\[\displaystyle {\it Landsberg}_{{3,2,2}}\, := \,0\]}
\end{maplelatex}
\mapleresult
\begin{maplelatex}
\mapleinline{inert}{2d}{Landsberg[3, 3, 1] := 0}{\[\displaystyle {\it Landsberg}_{{3,3,1}}\, := \,0\]}
\end{maplelatex}
\mapleresult
\begin{maplelatex}
\mapleinline{inert}{2d}{Landsberg[3, 3, 2] := 0}{\[\displaystyle {\it Landsberg}_{{3,3,2}}\, := \,0\]}
\end{maplelatex}
\mapleresult
\begin{maplelatex}
\mapleinline{inert}{2d}{Landsberg[3, 3, 3] := 0}{\[\displaystyle {\it Landsberg}_{{3,3,3}}\, := \,0\]}
\end{maplelatex}
\end{maplegroup}
\vspace{8pt}
\begin{maplegroup}
\begin{mapleinput}
\mapleinline{active}{1d}{Berwald[2, 2, 2] := simplify(diff(G2, y[2], y[2], y[2]))}{\[{\it Berwald}\, := \,{\it simplify} \left( {\frac {d^{3}}{d{{\it y}_{{2}}}^{3}}}{\it G2} \right) \]}
\end{mapleinput}
\mapleresult
\begin{maplelatex}
\mapleinline{inert}{2d}{}{}{\[\displaystyle {\it Berwald_{{2,2,2}}}\, := \frac{-3}{16}  {\frac {d \,f( {\it x1})}{d{\it x1}}}   {\it y}_{{2}}{{\it y}_{{3}}}^{3} (( 123286440\,{{\it y}_{{2}}}^{5}{{\it y}_{{3}}}^{5}{{\it y}_{{1}}}^{37}+6190070040\,{{\it y}_{{2}}}^{8}{{\it y}_{{3}}}^{8} {{\it y}_{{1}}}^{31}\]}
\end{maplelatex}
\begin{maplelatex}
\mapleinline{inert}{2d}{}{}{\[
+13029127584\,{{\it y}_{{2}}}^{9}{{\it y}_{{3}}}^{9} {{\it y}_{{1}}}^{29}+21263575256\,{{\it y}_{{2}}}^{13}{{\it y}_{{3}}}^{13} {{\it y}_{{1}}}^{21}+13029127584\,{{\it y}_{{2}}}^{14}{{\it y}_{{3}}}^{14} \]}
\end{maplelatex}
\begin{maplelatex}
\mapleinline{inert}{2d}{}{}{\[
{{\it y}_{{1}}}^{19}+6190070040\,{{\it y}_{{2}}}^{15}{{\it y}_{{3}}}^{15}
{{\it y}_{{1}}}^{17}+2252056776\,{{\it y}_{{2}}}^{7}{{\it y}_{{3}}}^{7} {{\it y}_{{1}}}^{33}+2576\,{{\it y}_{{2}}}^{22}{{\it y}_{{3}}}^{22}{{\it y}_{{1}}}^{3}\]}
\end{maplelatex}
\begin{maplelatex}
\mapleinline{inert}{2d}{}{}{\[
+1621224\,{{\it y}_{{2}}}^{20}{{\it y}_{{3}}}^{20} {{\it y}_{{1}}}^{7}+21263575256\,{{\it y}_{{2}}}^{10}{{\it y}_{{3}}}^{10} {{\it y}_{{1}}}^{27}+27114249960\,{{\it y}_{{2}}}^{12}{{\it y}_{{3}}}^{12} {{\it y}_{{1}}}^{23}\]}
\end{maplelatex}
\begin{maplelatex}
\mapleinline{inert}{2d}{}{}{\[
+27114249960\,{{\it y}_{{2}}}^{11}{{\it y}_{{3}}}^{11} {{\it y}_{{1}}}^{25}+2576\,{\it y}_{{2}}{\it y}_{{3}}{{\it y}_{{1}}}^{45} +{{\it y}_{{2}}}^{24}{{\it y}_{{3}}}^{24}+17363896\,{{\it y}_{{2}}}^{4}{{\it y}_{{3}}}^{4}{{\it y}_{{1}}}^{39} \]}
\end{maplelatex}
\begin{maplelatex}
\mapleinline{inert}{2d}{}{}{\[
+24\,{{\it y}_{{1}}}^{47}+91080\,{{\it y}_{{2}}}^{21}{{\it y}_{{3}}}^{21} {{\it y}_{{1}}}^{5}
+17363896\,{{\it y}_{{2}}}^{19}{{\it y}_{{3}}}^{19}
{{\it y}_{{1}}}^{9}+91080\,{{\it y}_{{2}}}^{2}{{\it y}_{{3}}}^{2}{{\it y}_{{1}}}^{43} ) \sqrt{{\it y}_{{2}}{\it y}_{{3}}}\]}
\end{maplelatex}
\begin{maplelatex}
\mapleinline{inert}{2d}{}{}{\[
+412896\,{{\it y}_{{2}}}^{21}{{\it y}_{{3}}}^{21}{{\it y}_{{1}}}^{6}+{{\it y}_{{1}}}^{48}+412896\,{{\it y}_{{2}}}^{3}{{\it y}_{{3}}}^{3}{{\it y}_{{1}}}^{42}
+300\,{{\it y}_{{2}}}^{23}{{\it y}_{{3}}}^{23}{{\it y}_{{1}}}^{2}+16974\,{{\it y}_{{2}}}^{22}{{\it y}_{{3}}}^{22}{{\it y}_{{1}}}^{4}\]}
\end{maplelatex}
\begin{maplelatex}
\mapleinline{inert}{2d}{}{}{\[
+5612805\,{{\it y}_{{2}}}^{20}{{\it y}_{{3}}}^{20}{{\it y}_{{1}}}^{8}
+48497064\,{{\it y}_{{2}}}^{19}{{\it y}_{{3}}}^{19}{{\it y}_{{1}}}^{10}+287134346\,{{\it y}_{{2}}}^{18}{{\it y}_{{3}}}^{18}{{\it y}_{{1}}}^{12}\]}
\end{maplelatex}
\begin{maplelatex}
\mapleinline{inert}{2d}{}{}{\[
+1222297740\,{{\it y}_{{2}}}^{17}{{\it y}_{{3}}}^{17}{{\it y}_{{1}}}^{14}+3864164634\,{{\it y}_{{2}}}^{16}{{\it y}_{{3}}}^{16}{{\it y}_{{1}}}^{16}
+9276875476\,{{\it y}_{{2}}}^{15}{{\it y}_{{3}}}^{15}{{\it y}_{{1}}}^{18}\]}
\end{maplelatex}
\begin{maplelatex}
\mapleinline{inert}{2d}{}{}{\[
+17172595110\,{{\it y}_{{2}}}^{14}{{\it y}_{{3}}}^{14}{{\it y}_{{1}}}^{20}+24755608584\,{{\it y}_{{2}}}^{13}{{\it y}_{{3}}}^{13}{{\it y}_{{1}}}^{22}+27948336381\,{{\it y}_{{2}}}^{12}{{\it y}_{{3}}}^{12}{{\it y}_{{1}}}^{24}\]}
\end{maplelatex}
\begin{maplelatex}
\mapleinline{inert}{2d}{}{}{\[
+24755608584\,{{\it y}_{{2}}}^{11}{{\it y}_{{3}}}^{11}{{\it y}_{{1}}}^{26}+17172595110\,{{\it y}_{{2}}}^{10}{{\it y}_{{3}}}^{10}{{\it y}_{{1}}}^{28}+9276875476\,{{\it y}_{{2}}}^{9}{{\it y}_{{3}}}^{9}{{\it y}_{{1}}}^{30}\]}
\end{maplelatex}
\begin{maplelatex}
\mapleinline{inert}{2d}{}{}{\[
+3864164634\,{{\it y}_{{2}}}^{8}{{\it y}_{{3}}}^{8}{{\it y}_{{1}}}^{32}
+1222297740\,{{\it y}_{{2}}}^{7}{{\it y}_{{3}}}^{7}{{\it y}_{{1}}}^{34}+287134346\,{{\it y}_{{2}}}^{6}{{\it y}_{{3}}}^{6}{{\it y}_{{1}}}^{36}\]}
\end{maplelatex}
\begin{maplelatex}
\mapleinline{inert}{2d}{}{}{\[
+300\,{\it y}_{{3}}{\it y}_{{2}}{{\it y}_{{1}}}^{46}+48497064\,{{\it y}_{{2}}}^{5}{{\it y}_{{3}}}^{5}{{\it y}_{{1}}}^{38}+5612805\,{{\it y}_{{2}}}^{4}{{\it y}_{{3}}}^{4}{{\it y}_{{1}}}^{40}
+16974\,{{\it y}_{{3}}}^{2}{{\it y}_{{2}}}^{2}{{\it y}_{{1}}}^{44}\]}
\end{maplelatex}
\begin{maplelatex}
\mapleinline{inert}{2d}{}{}{\[
+(615939264\,{{\it y}_{{2}}}^{6}{{\it y}_{{3}}}^{6}{{\it y}_{{1}}}^{35}
+123286440\,{{\it y}_{{2}}}^{18}{{\it y}_{{3}}}^{18}
{{\it y}_{{1}}}^{11}+615939264\,{{\it y}_{{2}}}^{17}{{\it y}_{{3}}}^{17}
{{\it y}_{{1}}}^{13}\]}
\end{maplelatex}
\begin{maplelatex}
\mapleinline{inert}{2d}{}{}{\[
+1621224\,{{\it y}_{{2}}}^{3}{{\it y}_{{3}}}^{3}{{\it y}_{{1}}}^{41}+2252056776\,{{\it y}_{{2}}}^{16}{{\it y}_{{3}}}^{16}
{{\it y}_{{1}}}^{15}+24\,{{\it y}_{{2}}}^{23}{{\it y}_{{3}}}^{23} {\it y}_{{1}}) \sqrt{{\it y}_{{2}}{\it y}_{{3}}} ) \]}
\end{maplelatex}
\begin{maplelatex}
\mapleinline{inert}{2d}{}{}{\[
/ (
 \sqrt{{\it y}_{{2}}{\it y}_{{3}}}f \left( {\it x1} \right)  ( 50\,{{\it y}_{{2}}}^{5}{{\it y}_{{3}}}^{5}{{\it y}_{{1}}}^{3}+126\,{{\it y}_{{2}}}^{3}{{\it y}_{{3}}}^{3}{{\it y}_{{1}}}^{7}+50\,{{\it y}_{{2}}}^{2}{{\it y}_{{3}}}^{2}{{\it y}_{{1}}}^{9}\]}
\end{maplelatex}
\begin{maplelatex}
\mapleinline{inert}{2d}{}{}{\[
+6\,{\it y}_{{2}}{\it y}_{{3}}{{\it y}_{{1}}}^{11}+6\,{{\it y}_{{2}}}^{6}{{\it y}_{{3}}}^{6}{\it y}_{{1}}+126\,{{\it y}_{{2}}}^{4}{{\it y}_{{3}}}^{4}{{\it y}_{{1}}}^{5}
+({{\it y}_{{2}}}^{6}{{\it y}_{{3}}}^{6} +21\,{{\it y}_{{2}}}^{5}{{\it y}_{{3}}}^{5} {{\it y}_{{1}}}^{2}\]}
\end{maplelatex}
\begin{maplelatex}
\mapleinline{inert}{2d}{}{}{\[
+90\,{{\it y}_{{2}}}^{4}{{\it y}_{{3}}}^{4}{{\it y}_{{1}}}^{4}+21\,{\it y}_{{2}}{\it y}_{{3}} {{\it y}_{{1}}}^{10}
+90\,{{\it y}_{{2}}}^{2}{{\it y}_{{3}}}^{2} {{\it y}_{{1}}}^{8}+141\,{{\it y}_{{2}}}^{3}{{\it y}_{{3}}}^{3} {{\it y}_{{1}}}^{6}
+ {{\it y}_{{1}}}^{12})\sqrt{{\it y}_{{2}}{\it y}_{{3}}} ) ^{4}
\]}
\end{maplelatex}
\end{maplegroup}
\vspace{8pt}
\begin{maplegroup}
\begin{mapleinput}
\mapleinline{active}{1d}{y[1] := 0; 1; y[2] := 1; 1; y[3] := 1; }{\[\]}
\end{mapleinput}
\mapleresult
\begin{maplelatex}
\mapleinline{inert}{2d}{y[1] := 0}{\[\displaystyle {\it y}_{{1}}\, := \,0\]}
\end{maplelatex}
\mapleresult
\begin{maplelatex}
\mapleinline{inert}{2d}{y[2] := 1}{\[\displaystyle {\it y}_{{2}}\, := \,1\]}
\end{maplelatex}
\mapleresult
\begin{maplelatex}
\mapleinline{inert}{2d}{y[3] := 1}{\[\displaystyle {\it y}_{{3}}\, := \,1\]}
\end{maplelatex}
\end{maplegroup}
\begin{maplegroup}
\begin{mapleinput}
\mapleinline{active}{1d}{simplify(Berwald[2, 2, 2])}{\[{\it simplify} \left( {\it Berwald} \right) \]}
\end{mapleinput}
\mapleresult
\begin{maplelatex}
\mapleinline{inert}{2d}{-(3/16)*(diff(f(x1), x1))/f(x1)}{\[\displaystyle \frac{-3}{16}\,{\frac {{\frac {d}{d{\it x1}}}f \left( {\it x1} \right) }{f \left( {\it x1} \right) }}\]}
\end{maplelatex}
\end{maplegroup}
\end{example}

\bigskip

 By Example \ref{ex.3}, for any non constant  positive smooth function $f$ on $\mathbb{R}$,    the Landsberg tensor of  $(M,F)$ vanishes (or equivalently, the hv-curvature $P$ of the Cartan connection vanishes)  and hence the class is Landsbergian. On the other hand, the hv-curvature \, $\overast{P}$ of the Chern  connection does not vanish and hence the class is not Berwaldian. So we can confirm:

 \begin{thm}
 There are non-regular Landsberg spaces which are not Berwaldian.
\end{thm}

\section*{Acknowledgement} The second author would like to express his deep  gratitude to Professors J\'{o}zsef Szilasi,  Zolt\'{a}n Muzsnay and Mr. D\'{a}vid Kert\'{e}sz (University of Debrecen) for their valuable discussions and comments.

\providecommand{\bysame}{\leavevmode\hbox
to3em{\hrulefill}\thinspace}
\providecommand{\MR}{\relax\ifhmode\unskip\space\fi MR }
\providecommand{\MRhref}[2]{%
  \href{http://www.ams.org/mathscinet-getitem?mr=#1}{#2}
} \providecommand{\href}[2]{#2}

\end{document}